\documentclass[12pt]{amsart} 

\usepackage[english]{babel}
\usepackage[latin1]{inputenc}
\usepackage{amsmath}
\usepackage{amsthm}
\usepackage{amssymb}
\usepackage{tikz}
\usepackage[all]{xy}
\usepackage{imakeidx}
\usepackage{appendix}
\usepackage{tikz-cd}
\usepackage{verbatim}
\usepackage{enumitem}
\usepackage{multicol}
\usepackage{csquotes}
\usepackage{silence,lmodern}

\usepackage{hyperref}
\usepackage{cleveref}

\usepackage{mathtools}
\newtheorem{thm}{Theorem}[section]

\newtheorem{prop}[thm]{Proposition}

\newtheorem{lemma}[thm]{Lemma}
\newtheorem{cor}[thm]{Corollary}
\newtheorem{conj}[thm]{Conjecture}

\theoremstyle{definition}
\newtheorem{defin}[thm]{Definition}

\theoremstyle{remark}
\newtheorem{rmk}[thm]{Remark}
\numberwithin{equation}{section}

\newcommand{\GL}{\operatorname{GL}}
\newcommand{\PGL}{\operatorname{PGL}}
\newcommand{\SL}{\operatorname{SL}}
\newcommand{\Sp}{\operatorname{Sp}}

\newcommand{\bbG}{\mathbb G}

\newcommand{\Gal}{\operatorname{Gal}}
\newcommand{\cha}{\operatorname{char}}

\newcommand{\Spec}{\operatorname{Spec}}
\newcommand{\Tors}{\operatorname{Tors}}

\newcommand{\trdeg}{\operatorname{trdeg}}

\renewcommand{\phi}{\varphi}
\newcommand{\ed}{\operatorname{ed}}

\newcommand{\too}{\longrightarrow}

\def\cftil#1{\ifmmode\setbox7\hbox{\accent"5E#1}\else\setbox7\hbox{\accent"5E#1}\penalty 10000\relax\fi\raise 1\ht7\hbox{\lower1.1ex\hbox to 1\wd7{\hss\accent"7E\hss}}\penalty 10000\hskip-1\wd7\penalty 10000\box7 }

\begin{document} 
\title[Special groups, versality and the Grothendieck-Serre Conjecture]{Special groups, versality and the Grothendieck-Serre Conjecture}

\author{Zinovy Reichstein}
\address{Department of Mathematics\\
 University of British Columbia\\ 
 Vancouver, BC V6T 1Z2\\Canada}
\email{reichst@math.ubc.ca}
\thanks{Reichstein was partially supported by National Sciences and Engineering Research Council of
 Canada Discovery grant 253424-2017.}
\author{Dajano Tossici}
\address{Univ. Bordeaux, CNRS, Bordeaux INP, IMB, UMR 5251,  F-33400, Talence, France}
\email{dajano.tossici@u-bordeaux.fr}
\thanks{Tossici was partially supported by the ANR CLap-CLap project, grant
ANR-18-CE40-0026-01 of the French Agence Nationale de la Recherche.}

	

\begin{abstract} Let $k$ be a base field and $G$ be an algebraic group over $k$. J.-P. Serre
defined $G$ to be special if every $G$-torsor $T \to X$ is locally trivial in the Zariski topology for every 
reduced algebraic variety $X$ defined over $k$. In recent papers an a priori weaker condition is used: $G$ is called 
special if every $G$-torsor 
$T \to \Spec(K)$ is split for every field $K$ containing $k$. We show that these two definitions are equivalent.
We also generalize this fact and propose a strengthened version of the Grothendieck-Serre conjecture based on the notion of essential dimension.
\end{abstract}

\maketitle

\section{Introduction}

Let $k$ be a base field and $G$ be an algebraic group (i.e., a group scheme of finite type) over $k$.  
Let $X$ be a $k$-scheme.  A morphism $T \to X$ is a pseudo $G$-torsor if $T$ is
 equipped with a (left) action of $G$ such that the mapping $G
 \times_X T \to T \times_X T$ given by $(g, x) \mapsto (x, g \cdot x)$
is an isomorphism.
A pseudo $G$-torsor $T$ is {\em a fppf $G$-torsor} 
if it is locally trivial in the 
fppf topology, i.e., if there exists a faithfully flat morphism $X' \to X$, of $k$-schemes,
locally finitely presented, such that $T \times_X X' \cong G \times_k X'$. 
We will denote the set of isomorphism classes of fppf $G$-torsors over $X$ by $\Tors(X,G)$. 
This set has a marked element, represented by the split fppf $G$-torsor $G \times_k X \to X$.

The pointed set $\Tors(X,G)$ is  contained in the \v{C}ech cohomology pointed set $H^1(X,G)$, computed 
in the fppf topology.
If $G$ is affine or if $G$ is smooth and $\dim(X) \leqslant 1$, then $\Tors(X,G)$ coincides with $H^1(X,G)$ (see \cite[Theorem 4.3 and Proposition 4.6]{milne}). 

If $A$ is a commutative $k$-algebra we will write $\Tors(A,G)$ and $H^1(A,G)$ in place of $\Tors(\Spec A, G)$  
and $H^1(\Spec A,G)$, respectively.   

One can also define \'etale $G$-torsors (resp. Zariski $G$-torsors) by replacing the fppf topology with the \'etale topology (resp. Zariski topology). Clearly a Zariski torsor is an \'etale torsor and an \'etale torsor is an fppf torsor.
For smooth algebraic groups fppf torsors and \'etale torsors coincide. In the sequel we will be primarily interested in
fppf torsors and will often abbreviate ``fppf $G$-torsor" to simply ``$G$-torsor".

In a foundational paper~\cite{serre-special} (reprinted in~\cite{serre-reprinted}), J.-P. Serre
defined a smooth algebraic group $G$ to be special if every $G$-torsor $T \to X $ over a reduced 
algebraic variety 
$X$ (i.e., a reduced separated scheme of finite type over $k$) is a Zariski torsor over $X$. 
His exact definition reads as follows: \textit{$G$ est sp\'ecial si tout fibr\'e principal de groupe $G$ est localement trivial}. In our terminology, 
`fibr\'e principal' means 
\'etale torsor (equivalently, `fppf torsor' since $G$ is smooth) and `localement trivial' 
means Zariski torsor. It is easy to see that this condition on $G$ is equivalent to the following
\begin{equation} \label{e.serre-special}
\text{$\Tors(R,G)=1$ for every local ring $R$ containing $k$,}
\end{equation}
Note that Serre assumed that the base field is algebraically closed, but his definition of special group remains valid and equivalent to~\eqref{e.serre-special} over any field; see Remark~\ref{rem.finite-type}.
Subsequently A.~Grothendieck classified special semisimple groups over an algebraically closed field; 
see~\cite[Theorem 3]{grothendieck-special}.

There has been renewed interest in special groups in recent years. However, many recent papers use an a priori different definition: they define an algebraic group $G$ 
to be special if $\Tors(K, G) = 1$ for every field $K$ containing $k$.\footnote{This condition forces $G$ to be smooth, affine and connected; see~Proposition~\ref{prop.non-linear}.} 
Some of these papers, e.g.,~\cite{reichstein} or~\cite{huruguen}, appeal to Grothendieck's classification, which is based on the Serre's definition of special group.
We will show that this discrepancy does not cause any problems because the classical and the modern definitions 
of special group are, in fact, equivalent. 

\begin{thm} \label{thm.main1} Let $G$ be an algebraic group defined over a field $k$. 
Then the following conditions are equivalent:

\smallskip
(1) $\Tors(  K, G)=1$ for every field $K$ containing $k$,

\smallskip
(2) $\Tors(  R, G) = 1$ for any local ring $R$ containing $k$,

\smallskip
(3) $\Tors(  S, G) = 1$ for any semi-local ring $S$ containing $k$.

\end{thm}

In the sequel, we will say that $G$ is ``(1)-special" if it satisfies (1), ``(2)-special" if it satisfies (2) and ``(3)-special" if it satisfies (3). 

\smallskip
The following conjecture arose in the above-mentioned classical papers;
see \cite[Section 5.5, Remark]{serre-special} and \cite[Remark 3, pp.~26-27]{grothendieck-special}.
It was presumably motivated by the difference between (1) and (2).

\begin{conj}[Grothendieck-Serre Conjecture] \label{conj.g-s}
Let $R$ be a regular local ring containing $k$ and $G$ be a smooth reductive algebraic group over $k$.
Then the natural morphism $H^1(R, G) \to H^1(K, G)$ has trivial kernel. 
\end{conj}

In the case, where $k$ is an infinite perfect field, Conjecture~\ref{conj.g-s} was proved 
by J.-L. Colliot-Th\'el\`ene and M.~Ojanguren~\cite[Theorem 3.2]{cto} for any smooth linear algebraic group
(not necessarily reductive). In the case, where $k$ is an arbitrary infinite field, it due 
to R.~Fedorov and I.~Panin~\cite{fedorov-panin}; in fact, Fedorov and Panin allow $R$ to be an arbitrary semi-local ring.
Panin~\cite{panin2, panin3} recently announced a proof in the case where $k$ is finite 
(also with $R$ an arbitrary semi-local ring)\footnote{The papers \cite{fedorov-panin, panin2, panin3} address
a stronger version of the Grothendieck-Serre conjecture, 
posed in~\cite[Remark 11.1a]{grothendieck-brauer2}, where $G$ is assumed to be a group scheme over $R$. 
We will only be interested in the ``constant case", where $G$ is defined over $k$.  To the best of our knowledge, 
for a finite base field $k$, even the constant case of the Grothendieck-Serre conjecture was open prior to Panin's work.}.

Our proof of Theorem~\ref{thm.main1} does not rely on the Grothendieck-Serre conjecture. In the case, 
where $k$ is infinite, we deduce it from Theorem~\ref{thm.main2} below. The proof of Theorem~\ref{thm.main2} 
is short and self-contained; see Section~\ref{sect.inf}. In the case where $k$ is finite, our proof of Theorem~\ref{thm.main1}
relies on recent work of M.~Huruguen; see Section~\ref{sect.finite}. In order to state Theorem~\ref{thm.main2} we shall need the following definition.

\begin{defin}
Let $G$ be a linear algebraic group over a field $k$. We will say that a $G$-torsor $\tau \colon V \to Y$ is weakly (1)-versal if
every $G$-torsor $\tau_1 \colon T_1 \to \Spec(K)$ over an infinite field $K$ containing $k$ can be obtained as a pull-back from $\tau$
via some morphism $\Spec(K) \to Y$. In other words, there exists a Cartesian diagram of $k$-morphisms
\[ \xymatrix{   T_1 \ar@{->}[d]_{\tau_1}  \ar@{->}[rr] &  & V \ar@{->}[d]^{\tau}  \\
                 \Spec(K)  \ar@{->}[rr] & & Y. }  
\]
Here $Y$ is an integral scheme of finite type over $k$. Similarly, we will say that the $G$-torsor $\tau \colon V \to Y$ is weakly (2)-versal (respectively, weakly (3)-versal)
if every $G$-torsor $\tau_2 \colon T_2 \to \Spec(R)$ (respectively, $\tau_3 \colon T_3 \to \Spec(S)$) over a local ring $R$ (respectively, a semi-local ring $S$)
containing $k$ can be obtained from $\tau$ by pull-back via a morphism $\Spec(R) \to Y$ (respectively, $\Spec(S) \to Y$).
Finally, for $n = 1, 2, 3$, we will say that $\tau \colon V \to Y$ is $(n)$-versal if the restriction of $\tau$ to every dense open subscheme of $Y$ is weakly $(n)$-versal.
\end{defin}

The notion of (1)-versality was studied in~\cite{serre-gc} and~\cite{dr} under the name of ``versality".

\begin{thm} \label{thm.main2} Let $G$ be a linear algebraic group over an infinite field $k$, $Y$ be an integral scheme of finite type over $k$,
and $\tau \colon V \to Y$ be a $G$-torsor. Then 

\smallskip
(a) $\tau$ is weakly (1)-versal $\Longleftrightarrow$ $\tau$ is weakly (2)-versal $\Longleftrightarrow$ $\tau$ is weakly (3)-versal. 

\smallskip
(b) $\tau$ is (1)-versal $\Longleftrightarrow$ $\tau$ is (2)-versal $\Longleftrightarrow$ $\tau$ is (3)-versal. 
\end{thm}

Let $L$ be a field containing $k$ and $\mu \colon T \to \Spec(L)$
be a $G$-torsor. We will say that $\mu$ descends to an intermediate subfield $k \subset L_0 \subset L$ if
$\mu$ is the pull-back of some $G$-torsor $\mu_0 \colon T_0 \to \Spec(L_0)$, i.e., if there exists a Cartesian diagram 
of the form
 \[ \xymatrix{ T \ar@{->}[r] \ar@{->}[d]^{\mu} & T_0 \ar@{->}[d]^{\mu_0}   \\
 \Spec(L) \ar@{->}[r]  &  \Spec(L_0). } \]
The essential dimension $\ed(\mu)$ of $\mu$ is the smallest value of the transcendence degree $\trdeg(L_0/k)$
such that $\mu$ descends to $L_0$. 
The essential dimension $\ed(G)$ of $G$ is the maximal value of
$\ed(\mu)$, as $K$ ranges over all fields containing $k$ and $\tau$ ranges over all $G$-torsors $T \to \Spec(K)$.
Sometimes we will write $\ed_k(\mu)$ in place of $\ed(\mu)$ to emphasize that this number depends on the base field $k$, and similarly
for $\ed_k(G)$. 
Note that $G$ is (1)-special if and only if $\ed(G) = 0$; see~Corollary~\ref{cor.special-ed}. 
For a detailed discussion of essential dimension and further references, see~\cite{reichstein-icm} or~\cite{merkurjev-survey}.

In Section~\ref{sect.cor.main3} we will prove the following corollary of Theorem~\ref{thm.main2}.

\begin{cor} \label{cor.main3} Let $k$ be an infinite field, $G$ be a linear algebraic group over $k$ of essential dimension $d$,
$S$ be a semi-local ring containing $k$ and $\tau \colon T \to \Spec(S)$ be a $G$-torsor. Then 
there exists a Cartesian diagram of $k$-morphisms
\[ \xymatrix{   T \ar@{->}[d]_{\tau}  \ar@{->}[rr] &  & W \ar@{->}[d]^{\nu}  \\
                 \Spec(S)  \ar@{->}[rr] & & Y, }  
\]
where $Y$ is a $d$-dimensional geometrically integral scheme of finite type over $k$, $\nu$ is a $G$-torsor and $W(k) \neq \emptyset$. 
\end{cor}


In a similar spirit, we would like to propose the following variant of the Grothendieck-Serre conjecture.

\begin{conj} \label{conj.main4}
Let $k$ be an algebraically closed field, $G$ be a connected reductive linear algebraic group over $k$, $R$ be a regular local ring containing $k$, and
$\tau \colon T \to \Spec(R)$ be a $G$-torsor. Let $K$ be the field of fractions of $R$ and 
$\tau_K \colon T_K \to \Spec(K)$ be the $G$-torsor 
obtained by restricting $\tau$ to the generic point of $\Spec(R)$. Assume that $\ed_k (\tau_K) = d$. Then 
there exists a Cartesian diagram of $k$-morphisms
\[ \xymatrix{   T \ar@{->}[d]_{\tau}  \ar@{->}[rr] &  & W \ar@{->}[d]^{\nu}  \\
                 \Spec(R)  \ar@{->}[rr] & & Y, }  
\]
where $Y$ is a $d$-dimensional integral scheme of finite type over $k$
and $\nu$ is a $G$-torsor. 
\end{conj}

One may also consider stronger versions of Conjecture~\ref{conj.main4},
where $R$ is allowed to be semi-local, $G$ is not required to be reductive, and/or 
the assumption on the base field is weakened (e.g., $k$ is only assumed to be infinite or perfect, or perhaps, allowed to be an arbitrary field). 
If $k$ is not assumed to be algebraically closed, then it makes sense to also ask that $W(k)$ should be non-empty and $Y$ geometrically integral, as in Corollary~\ref{cor.main3}, 
so that Conjecture~\ref{conj.main4} reduces to the Grothendieck-Serre Conjecture~\ref{conj.g-s} when $d=0$. 
We do not know how to prove or disprove any of these versions. Some (admittedly modest) evidence for Conjecture~\ref{conj.main4} is presented in Section~\ref{sect.conj.main4}.

\section{Preliminaries on (1)-special groups}\label{sect.1spec}

Throughout this paper $G$ will denote an algebraic group defined over a base field $k$. Unless otherwise specified, we will
not assume that $G$ is linear. We will use the terms ``linear" and ``affine" interchangeably in reference to algebraic groups.

\begin{lemma} \label{lem.long}
Let $X$ be a scheme over $k$.
Let $G_1 \hookrightarrow G$ be a closed immersion of algebraic groups over $k$. 
Then the natural sequence of pointed sets
$$
1 \too G_1(X)\too G(X)\too (G_1 \! \! \setminus \! G)(X)\stackrel{}{\too} H^1(X,G_1)\stackrel{}{\too} H^1(X,G) 
$$
is exact for any $k$-scheme $X$. 
\end{lemma}

Here $G_1 \! \! \setminus \! G$ denotes the homogeneous space parametrizing the right cosets of $G_1$ in $G$.

\begin{proof} 
See \cite[Proposition III \S 4, 4.6]{DG} and \cite[Proposition III 4.6]{milne}.
\end{proof}

\begin{rmk}
As a consequence of Lemma~\ref{lem.long}, we see that any element of the kernel of the map $H^1(X,G_1)\stackrel{}{\too} H^1(X,G) $ belongs to $\Tors(X,G)$. Indeed the map \[ (G_1 \! \! \setminus \! G)(X)\stackrel{}{\too} H^1(X,G_1) \] send a morphism $f:X \to G_1\! \! \setminus \! G$  to the pull-back of the $G_1$-torsor $G \to G_1 \! \setminus \! G$ via $f$.
\end{rmk}

We now proceed with the main result of this section.

\begin{prop} \label{prop.non-linear}
Let $G$ be a (1)-special algebraic group over a field $k$. Then 

\smallskip (a)
$G$ is smooth,

\smallskip
(b) $G$ is linear, 

\smallskip
(c) $G$ is connected.
\end{prop}

Our proof of parts (b) and (c) below is adapted from~\cite[Section 4.1]{serre-special}, where (2)-special 
groups are shown to be linear and connected.

\begin{proof}
(a)  By~\cite[Theorem 1.2]{tossici-vistoli}, $\ed(G) \geqslant \dim(\mathcal{G}) - \dim(G)$, where $\mathcal{G}$ is the Lie algebra of $G$ \footnote{This 
 inequality is stated in~\cite[Theorem 1.2]{tossici-vistoli} only in the case where $G$ is linear. However, the proof given there goes through for any algebraic group $G$.}.
If $G$ is (1)-special then, clearly, $\ed(G) = 0$ and this inequality tells us that $\dim(\mathcal{G}) = \dim(G)$. This shows that $G$ is smooth.

\smallskip
(b)  By faithfully flat descent, \cite[Proposition 2.6.1]{EGAIV2}, we may (and will) assume that $k$ is algebraically closed. 
By part (a), $G$ is smooth. As we explained in the Introduction, when $G$ is smooth, 
$\Tors(K, G)$ coincides with $H^1(K, G)$ for any field $K$ containing $k$. Consequently, 
$H^1(K,G)=0$ for any such $K$.  We now proceed in two steps.

\smallskip
{\bf Step 1.} Assume that $G$ is connected. By Chevalley's structure theorem~\cite{chevalley, Con} there exists 
a unique connected normal linear $k$-subgroup of $G$ such that the quotient is an abelian variety $A$. 

We claim that $A$ is trivial.  
Assume the contrary. Then by \cite[Lemma 3]{serre-special}, there exists a cyclic subgroup 
$C$ of $G$ of prime order $l$, distinct from the characteristic of $k$, 
such that the composition $C \too G \too A$
is injective. Let $K = k(t)$, where $t$ is an indeterminate.
 By Lemma \ref{lem.long}, the inclusion $C \hookrightarrow A$ induces an exact sequence of pointed sets
 $$ A(K)\too (C \! \setminus \! A)(K)\too H^1(K,C)\too H^1(K, A). $$
 Note that $C \! \setminus \! A$ is an abelian variety. Hence,
 every rational map $\mathbb{A}^1 \dasharrow C \! \setminus \! A$ is constant; see, e.g.,~\cite[Proposition 3.9]{Mi3}.
 Consequently, $(C \! \setminus \! A)(K)=(C \! \setminus \! A)(k)$ and thus the morphism $A(K) \to (C \! \setminus \! A)(K)$ is surjective.
 We conclude that 
 \begin{equation} \label{e.kernel}
 \text{the morphism $ H^1(K,C)\too H^1(K,A)$ has trivial kernel.}
 \end{equation}
 Now recall that the inclusion $C \hookrightarrow A$ factors through $G$. Thus the morphism $ H^1(K,C)\to H^1(K,A)$ factors through $H^1(K, G)$.
 Since we are assuming that $G$ is (1)-special and thus $H^1(K,G)= 1$, we conclude that $H^1(K, C) \to H^1(K, A)$ is the trivial map. That is,
 the kernel of this map is all of $H^1(K, C)$. Now~\eqref{e.kernel} tells us that 
 $H^1(K, C) = 1$.  On the other hand, since $l$ is 
 different from the characteristic of $k$, $C$ is isomorphic to $\mu_l$ and by Kummer theory,
 $H^1(K,C)\simeq K^*/(K^*)^l \neq 1$, a contradiction.
 
 \smallskip
 {\bf Step 2.} Now let $G$ be an arbitrary (1)-special group over $k$. By the definition of essential dimension,
 $\ed(G)=0$. Denote the identity component of $G$ by $G^0$. 
 Since $G^0$ is a closed subgroup of $G$ of finite index, $\ed(G^0) \leqslant \ed(G)$ (see~\cite[Principle 2.10]{Bro}) and thus
 $\ed(G^0) = 0$. Since $k$ is algebraically closed we conclude that $G^0$ is (1)-special and hence,
 affine by Step 1. Since $k$ is algebraically closed, every connected component of $G$ has a $k$-rational point. 
 Consequently, every connected component is isomorphic to $G^0$ (as a variety). Thus
 $G$ is the disjoint union of finitely many affine varieties (each isomorphic to $G^0$). We conclude that $G$ is affine, and hence linear.

\smallskip
(c) By part (a), $G$ is a closed subgroup of $\GL_n$ for some $n \geqslant 1$.
The natural projection $\pi \colon \GL_n \to X = G \! \setminus \! \GL_n$ is then a $G$-torsor. Clearly $X$ is integral.
Since we are assuming that $G$ is (1)-versal, $\pi$ splits over the generic point of $X$. 
Consequently, $\GL_n$ is birationally isomorphic to $G \times X$. 
Since $\GL_n$ is connected, we conclude that $G$ is also connected.
\end{proof}

\begin{cor} \label{cor.special-ed}
Let $G$ be an algebraic group over a field $k$ (not necessarily affine).
Then $G$ is (1)-special if and only if $\ed(G) = 0$.
\end{cor}

\begin{proof} (a) The implication 
\[ \text{$G$ is (1)-special $\Longrightarrow$ $\ed(G) = 0$} \]
follows immediately from the definition of essential dimension.
If $k$ is algebraically closed, the converse is also obvious (we have already used this observation in the proof of Step 2 above).

Now assume that $k$ is an arbitrary field and $\ed(G) = 0$.
Then clearly $\ed(G_{\overline{k}}) = 0$, where $G_{\overline{k}} = G \times_{\Spec(k)} \Spec(\overline{k})$
and $\overline{k}$ denotes the algebraic closure of $k$. As we pointed out above, this implies that $G_{\overline{k}}$
is (1)-special. By Proposition~\ref{prop.non-linear}(b), $G_{\overline{k}}$ is affine. By faithfully flat descent, 
$G$ is also affine. For an affine group $G$, a proof of the implication
\[ \text{$\ed(G) = 0$ $\Longrightarrow$ $G$ is (1)-special} \]
can be found in~\cite[Proposition 4.4]{merkurjev} or~\cite[Proposition 4.3]{tossici-vistoli}.
\end{proof}

\section{Preliminaries on (3)-special groups}
%
%

The following lemma will be repeatedly used in the sequel.

\begin{lemma} \label{lem.torus1} (a) Suppose $1 \to G_1 \to G \to G_2 \to 1$ is an exact sequence of algebraic groups
defined over $k$. If $G_1$ and $G_2$ are (3)-special, then so is $G$.

\smallskip
(b) If $G = G_1 \times_k G_2$ is a direct product of $G_1$ and $G_2$, then the converse holds as well:
$G$ is (3)-special if and only if both $G_1$ and $G_2$ are (3)-special.

\smallskip
(c) Let $l/k$ be a field extension of finite degree and $G$ be an algebraic group defined over $l$. If $G$ is (3)-special over $l$,
then the Weil restriction $R_{l/k}(G)$ is (3)-special over $k$.
\end{lemma}

\begin{proof}  Throughout the proof, $S$ will denote a semi-local ring containing $k$.

\smallskip
(a) 
%
%
The exact sequence 
$1 \to G_1 \to G \to G_2 \to 1$ of algebraic groups over $k$ gives rise to an exact sequence
\begin{equation} \label{e.long(b)}
1 \to G_1(S) \to G(S) \to G_2(S) \to H^1(S, G_1) \to H^1(S, G) \to H^1(S, G_2);
\end{equation}
see~\cite[Section III.4]{milne}.
Since $G_1$ and $G_2$ are (3)-special they are (1)-special. By Proposition~\ref{prop.non-linear}(b), $G_1$ and $G_2$ are linear. Hence we have $H^1(S, G_1) = \Tors (S,G_1)=1$ and  $H^1(S, G_2)= \Tors(S,G_2) = 1$. We conclude that $H^1(S, G) = 1$, as desired.

\smallskip
(b) Suppose $G$ is (3)-special. Then $G$ is (1)-special and hence, affine by Proposition~\ref{prop.non-linear}(b). Since $G_1$ and $G_2$ are isomorphic to closed subgroups of $G$, they are also affine. Now $H^1(S, G) = H^1(  S, G_1) \times H^1(  S, G_2)$ by \cite[III \S 4, 4.2]{DG}. Since $H^1(S, G)$ is trivial, so are $H^1(S, G_1)$ and $H^1(S, G_2)$.

\smallskip
(c) Since $G$ is (3)-special, it is also (1)-special, and hence, linear by Proposition~\ref{prop.non-linear}(b). 
By the Faddeev-Shapiro theorem, $H^1(S, R_{l/k}(G_l)) = H^1(S \otimes_k l, G)$. Note that $S \otimes_k l$ is a semi-local ring
containing $l$. Since $G$ is (3)-special over $l$, $H^1(S \otimes_k l, G) = 1$, and part (c) follows.
\end{proof}

Recall that a smooth connected algebraic $k$-group $U$ is called unipotent if over the algebraic closure $\overline{k}$ there exists a tower of algebraic groups 
\begin{equation} \label{e.tower} 
 \xymatrix{  1 = U_0  \ar@{^{(}->}[r] & U_1 \ar@{^{(}->}[r]  &  \ldots \ar@{^{(}->}[r] & U_{r-1} \ar@{^{(}->}[r] & U_r = U   } 
\end{equation}
such that each $U_i$ is normal in $U_{i+1}$ and the quotient $U_{i+1}/U_i$ is isomorphic to $\bbG_a$ (over $\overline{k})$. A unipotent group $U$ over $k$ 
is called split, if there is a tower~\eqref{e.tower} such that the subgroups $U_i$ and the isomorphisms $U_{i+1}/U_i \simeq \bbG_a$ are all defined over $k$.

\begin{lemma} \label{lem.examples} The following algebraic groups are (3)-special for every positive integer $n$:

\smallskip
(a) the general linear group $\GL_n$, 

\smallskip
(b) the special linear group $\SL_n$, 

\smallskip
(c) the symplectic group $\Sp_{2n}$, 

\smallskip
(d) any $k$-split smooth connected unipotent group.
\end{lemma}

Our proof of Lemma~\ref{lem.examples} is similar to the arguments in~\cite[Section 4.4]{serre-special}, where the same groups are shown to be (2)-special.

\begin{proof}  Let $S$ be a semi-local ring containing $k$.

\smallskip
(a) Elements of $H^1(S, \GL_n)$ are in a natural bijective correspondence with isomorphism classes of projective modules 
of rank $n$ over $S$; see, e.g.,~\cite[III.(2.8)]{knus}. Here a projective $S$-module $M$ is said to be of rank $n$ if
$M \otimes_S  (S/I)$ is an $n$-dimensional vector space over $S/I$ for every maximal ideal $I$ of $S$. 
Part (a) is thus a restatement of~\cite[Lemma 1.4.4]{bruns-herzog}: every projective module of rank $n$ over a semi-local ring is free.

\smallskip
(b) By~\eqref{e.long(b)}, the exact sequence of algebraic groups 
\[ \xymatrix{  1 \ar@{->}[r] & \SL_n \ar@{->}[r]  & \GL_n \ar@{->}[r]^{\det} & \bbG_m \ar@{->}[r] & 1   } \]
induces an exact sequence
\[ \xymatrix{  \GL_n(S) \ar@{->}[r]^{\det} & \bbG_m(S) \ar@{->}[r]  & H^1(S, \SL_n) \ar@{->}[r] & H^1(S, \GL_n)  } \] 
in cohomology. By part (a), $H^1(S, \GL_n) = 1$. Thus
in order to show that $H^1(S, \SL_n) = 1$ it suffices to show that the map $\det \colon \GL_n(S) \to \bbG_m(S)$ is surjective. On the other hand, 
the surjectivity of this map follows from the fact that
\[ \det \begin{pmatrix} a & 0 & \ldots & 0 \\
                         0 & 1 & \ldots & 0 \\
                         \hdotsfor{4} \\
                         0 & 0 & \ldots & 1 \end{pmatrix} = a \]
for any $a \in \bbG_m(S)$.

\smallskip
(c) $H^1(S, \Sp_{2n})$ is in a natural bijective correspondence with isomorphism classes of projective $S$-modules $M$ 
of rank $2n$, equipped with a symplectic form; see, e.g.,~\cite[III. (2.5.1)]{knus}. As we saw in part (a), every projective module over a semi-local ring is free, 
$M \simeq S^{2n}$. Moreover, up to isomorphism, there is only one symplectic form on $S^{2n}$, $x_1 \wedge x_2 + \dots + x_{2n-1} \wedge x_{2n}$;
see, e.g., \cite[Proposition 2.1]{kirkwood}.
Thus $H^1(S, \Sp_{2n}) = 1$, as claimed.

\smallskip
(d) Applying Lemma~\ref{lem.torus1}(a) to the tower~\eqref{e.tower} recursively, we reduce to the case, where $U = \bbG_a$.
In this case part (d) follows by \cite[Proposition III 3.7]{milne} which states that fppf cohomology is the same as Zariski cohomology for coherent sheaves.
Note that $\Spec(S)$ is an affine scheme, and Zariski cohomology of a quasi-coherent sheaf over an affine scheme is trivial.  
\end{proof}

\begin{rmk} \label{rmk.tan} Combining Lemma~\ref{lem.examples}(d) with a theorem of N.~D.~T\^{a}n~\cite{tan}, we see that for a smooth connected unipotent group $U$ defined over $k$ the following conditions are equivalent: (a) $U$ is (1)-special, (b) $U$ is (2)-special, (c) $U$ is (3)-special, and (d) $U$ is split. We shall not need this in the sequel.
\end{rmk}

We are now in a position to prove Theorem~\ref{thm.main1} in the case, where $G$ is a torus.

\begin{lemma}  \label{lem.torus2} Let $T$ be a torus over $k$. If $T$ is (1)-special, then $T$ is (3)-special.
\end{lemma}

Recall that a torus $T$ over $k$ is called quasi-trivial if its character $\Gal(k)$-lattice is a permutation lattice.
Equivalently, $T$ is quasi-trivial if and only if $T = R_{l/k}(\bbG_m)$ for some finite field extension $l/k$.

\begin{proof}[Proof of Lemma~\ref{lem.torus2}] By a theorem of Colliot-Th\'el\`ene's, $T$ is (1)-special if and only if it is a direct factor of a quasi-trivial torus; see~\cite[Theorem 18]{huruguen}. In other words, there exists another torus $T'$ over $k$ such that $Q = T \times T'$ is a quasi-trivial torus. As we mentioned above, every quasi-trivial torus $Q$ over $k$ is of the form $Q = R_{l/k}(\bbG_m)$ for some finite field extension $l/k$. By Lemma~\ref{lem.examples}(a), $\bbG_m = \GL_1$ is (3)-special. Hence, by Lemma~\ref{lem.torus1}(c), $Q$ is (3)-special, and by Lemma~\ref{lem.torus1}(b), $T$ is (3)-special.
\end{proof}

\section{Proof of Theorem~\ref{thm.main2}}
\label{sect.inf}

Our proof will rely on the following lemma.

\begin{lemma}\label{lem.gln} 
Let $\Gamma$ be a smooth connected algebraic group over $k$, $G$ be a closed subgroup also defined over $k$, 
$S$ be a semi-local ring containing $k$, and $\tau: T\to \Spec S$ be a $G$-torsor. Assume that $\tau$ lies in the kernel 
of the natural map $H^1(S, G) \to H^1(S, \Gamma)$. Then

\smallskip
(a) there exists a $G$-equivariant morphism $f \colon T\to \Gamma$.

\smallskip
(b) Moreover, assume that $\Gamma(k)$ is dense in $\Gamma$. Then for any for any non-empty $G$-invariant 
open subvariety $U \subset \Gamma$ defined over $k$ there exists a $G$-equivariant morphism $f \colon T\to U$.
\end{lemma}

\begin{rmk}\label{rmk.3special} If $\Gamma$ is (3)-special, then  $\Gamma$ is affine by Proposition \ref{prop.non-linear}(b),  so $\Tors(S, \Gamma) =H^1(S,G)= 1$, and Lemma~\ref{lem.gln}
applies to every $G$-torsor $\tau$ over a semi-local ring $S$. In this case part (a) of the Lemma is equivalent to the assertion that
the (left) $G$-torsor $\pi \colon \Gamma \to G \!  \setminus \! \Gamma$ is weakly (3)-versal, and part (b) is equivalent 
to the assertion that $\pi$ is (3)-versal.
\end{rmk}

\begin{rmk} \label{rmk.unirational}
Assume $\Gamma$ is linear and $k$ is infinite, and furthermore, $\Gamma$ is reductive or $k$ is perfect. 
Then the condition that $\Gamma(k)$ is dense in $\Gamma$ in part (b) of the lemma
is automatic because $\Gamma$ is unirational over $k$; see \cite[Theorem 18.2]{borel}.
\end{rmk}

\begin{proof}[Proof of Lemma~\ref{lem.gln}] (a) Let $X = G \!  \setminus \! \Gamma$. 
The natural projection $\pi \colon \Gamma \to X$ is a $G$-torsor and $\pi(U)$ is a dense open subvariety of $X$.
By Lemma~\ref{lem.long}, $\tau$ lies in the image of the morphism $X(S) \to H^1(S, G)$. This means that 
$\tau$ is the pull-back of $\pi$ via a morphism $\overline{\alpha} \colon \Spec(S) \to X$. In other words, 
there exists a Cartesian diagram 
\[ \xymatrix{  T \ar@{->}[r]^{\alpha} \ar@{->}[d]_{\tau} &   \Gamma \ar@{->}[d]_{\pi}  \\
         \Spec(S) \ar@{->}[r]^{\quad \overline{\alpha}}  & X. }  
\]
Here ${\alpha} \colon T \to \Gamma$ is a $G$-equivariant morphism. 

(b) Set $Z = X \setminus \pi(U)$. Let $X_1, \dots, X_m \subset X$ denote
the Zariski closures of the images of the closed points of $S$ under $\overline{\alpha}$ (with the reduced scheme structure). Note that $\Gamma$ acts on $X = G \! \setminus \! \Gamma$ by right translations.

We claim that there exists a $g \in \Gamma(k)$ such that $g(X_i) \not \subset Z$ for every $i = 1, \dots, m$.  
If we can prove this claim, then the composition $f = t_g \circ \alpha \colon T \to \Gamma$ is a $G$-equivariant morphism 
and its image lies in $U$, as desired. Here
$t_g \colon \Gamma \to \Gamma$ denotes right multiplication by $g^{-1}$, $t_g(\gamma) = \gamma \cdot g^{-1}$. 

It remains to prove the claim. Denote the irreducible components of $Z$ by $Z_1, \ldots, Z_n$.
Then $g(X_i) \not \subset Z$ if and only if $g(X_i) \not \subset Z_j$ for any $j = 1, \dots, n$.
The points $g \in \Gamma(k)$ such that $g(X_i) \subset Z_j$ are the $k$-points of a closed 
subvariety $\Lambda_{ij} \subset \Gamma$. Since $\Gamma(k)$ is dense in $\Gamma$, it suffices to show that 
\[
\bigcup_{i, j} \Lambda_{ij} \neq \Gamma, 
\]
where the union is taken over all $i = 1, \ldots, m$ and $j = 1, \ldots, n$. Equivalently, it suffices to show that
\begin{equation} \label{e.lambda}
\text{$\Lambda_{ij} \neq \Gamma$ for every $i = 1, \ldots, m$ and $j = 1, \ldots, n$.}
\end{equation}
For the purpose of proving~\eqref{e.lambda}, we may pass to the algebraic closure of $k$ 
and thus assume that $k$ is algebraically closed.
By Kleiman's Tranversality Theorem~\cite[Theorem 2]{kleiman} there is a dense open 
subvariety $O_{ij} \subset \Gamma$ such that $g(X_i)$ intersects
$Z_j$ transversely for any $g \in O_{ij}(k)$. Since $Z_j \neq X$, this 
implies that $g(X_i) \not \subset Z_j$ for any $g \in O_{ij}(k)$.
Thus $\Lambda_{ij}$ lies in the complement of $O_{ij}$ in $\Gamma$, and~\eqref{e.lambda} follows.
This completes the proof of the claim and thus of part (b). 
\end{proof}

\begin{rmk} \label{rem.finite-type}
As a consequence of Lemma~\ref{lem.gln}(a), we see that Serre's original definition of special 
group, which was discussed in the Introduction, is equivalent to our notion of (2)-special group. 
In other words, the following conditions on an algebraic group $G$ defined over a field $k$ are equivalent:

\smallskip
(2) $H^1(R, G) = 1$ for any local ring $R$ containing $k$, and

\smallskip
(2$'$) $G$ is smooth and every $G$-torsor $\pi \colon Y \too X$, where $X$ is a reduced algebraic variety over $k$,  is a Zariski torsor, i.e., is locally trivial in the Zariski topology.

\smallskip
Proof of the implication (2) $\Longrightarrow$ (2$'$). Assume that $G$ satisfies (2). Then $G$ is smooth 
by Proposition \ref{prop.non-linear}(a), and (2$'$) readily follows.

\smallskip
Proof of the implication (2$'$) $\Longrightarrow$ (2). Assume that $G$ satisfies (2$'$). We claim that $G$ is affine. To prove this claim we may pass to the algebraic closure of $k$, i.e., assume that $k$ is algebraically closed. In this case $G$ is linear by~\cite[Theorem 1]{serre-special}. This proves the claim. We conclude that $G$ can be embedded as a closed subgroup scheme 
of $\GL_n$, for some $n$. Since the homogeneous space $G \! \setminus \! \GL_n$ is a reduced algebraic variety over $k$, and $G$ satisfies (2$'$), $\pi \colon \GL_n \too G \! \setminus \! \GL_n$ is a Zariski torsor.
Now consider an arbitrary $G$-torsor $\tau: Y \too \Spec R$, where $R$ a local ring containing $k$. Our goal is to show that
$\tau$ is split. By Lemma~\ref{lem.examples}(a) and Remark~\ref{rmk.3special}, 
$\pi \colon \GL_n \too G \! \setminus \! \GL_n$ is a weakly versal $G$-torsor.
In particular, $\tau$ is obtained from $\pi$ by pull-back via some morphism $\Spec R \too G \! \setminus \! \GL_n$.  
Since $\pi$ is locally trivial in the Zariski topology, this tells us that $\tau$ is split.
\qed
\end{rmk}

We are now ready to finish the proof of Theorem~\ref{thm.main2}. Part (b) is an immediate consequence of (a) and the definition of versality.
In part (a), the implications 
\[ \text{$\tau$ is weakly (3)-versal} \Longrightarrow 
 \text{$\tau$ is weakly (2)-versal} \Longrightarrow
  \text{$\tau$ is weakly (1)-versal} \]
 are obvious. So, we will assume that $\tau \colon V \to Y$ is a weakly (1)-versal $G$-torsor and will aim to show that
 $\tau$ is weakly (3)-versal.
 
Recall that we are assuming that $G$ is a linear algebraic group, i.e., a closed subgroup of $\GL_n$ for some $n \geqslant 1$.
Set $X = G \! \setminus \! \GL_n$, let $\pi \colon \GL_n \to X$ be the natural projection and 
$\eta$ be the generic point of $X$. Since $\tau$ is weakly (1)-versal, $\pi_{\eta}$ is the pull-back of $\tau$. 
That is, over some dense open subvariety $X_0 \subset X$ defined over $k$, 
$\pi$ is the pull-back of $\tau$, via a Cartesian diagram 
\[ \xymatrix{   \GL_n \ar@{->}[d]_{\pi}  & U_0 \ar@{->}[d]_{\pi} \ar@{_{(}->}[l]_{\text{open}}  \ar@{->}[rr]^{\phi} &  & V \ar@{->}[d]_{\tau'}  \\
 G \! \setminus \! \GL_n & \ar@{_{(}->}[l]_{\text{open}}  X_0 \ar@{->}[rr]^{\overline{\phi}} & & Y. }  
\]
Here $U_0 = \pi^{-1}(X_0)$. Now suppose $S$ is a semi-local ring containing $k$ and $\tau_3 \colon T_3 \to \Spec(S)$ be a $G$-torsor. 
(Here the ``3" in the subscript indicates that we are testing for (3)-versality.)
By Lemma~\ref{lem.examples}(a), $\GL_n$ is (3)-special. Moreover, since $k$ is infinite, $\GL_n(k)$ is dense in $\GL_n$. 
Applying Lemma~\ref{lem.gln} with $\Gamma = \GL_n$,
we conclude that there exists a $G$-equivariant morphism $f \colon T_3 \to U_0$. 
Composing $f$ and $\phi$ we obtain a Cartesian diagram
\[ \xymatrix{  T_3 \ar@{->}[r]^{f} \ar@{->}[d]_{\tau_3} &  U_0 \ar@{->}[d]_{\pi} \ar@{->}[rr]^{\phi} & & V \ar@{->}[d]_{\tau}  \\
         \Spec(S) \ar@{->}[r]^{\quad \overline{f}} &  X_0 \ar@{->}[rr]^{\overline{\phi}} & & Y }  
\]
which shows that $\tau$ is (3)-versal. 
\qed

\section{Proof of Theorem~\ref{thm.main1}}
\label{sect.finite}

The implications (3) $\Longrightarrow$ (2) $\Longrightarrow$ (1) are obvious, so we will focus on showing that (1) $\Longrightarrow$ (3).
Suppose $G$ is (1)-special. Then by Proposition~\ref{prop.non-linear}, $G$ is smooth, linear and connected. Our goal is to show that $G$ is (3)-special.
We will consider three cases.

\smallskip
{\bf Case 1:} The base field $k$ is infinite. Since $G$ is (1)-special, the trivial torsor $\pi \colon G \to \Spec(k)$ is weakly (1)-versal. By Theorem~\ref{thm.main2},
$\pi$ is also weakly (3)-versal. In other words, for any local ring $S$ containing $k$ and any $G$-torsor $\mu \colon T \to \Spec(S)$, there exists a Cartesian diagram
\[ \xymatrix{   T \ar@{->}[d]_{\mu}  \ar@{->}[rr] &  & G \ar@{->}[d]^{\pi}  \\
                 \Spec(S)  \ar@{->}[rr] & & \Spec(k), }  
\]
We conclude that $\mu$ is split. Thus shows that $G$ is (3)-special, as desired.

\smallskip
{\bf Case 2:} $G$ is a connected reductive group defined over a finite field $k$. By~\cite[Proposition 16.6]{borel}, a reductive group over a finite field $k$ 
is quasi-split, i.e., has a Borel subgroup defined over $k$. This allows us to appeal to the following result, due to M.~Huruguen~\cite[Proposition 15]{huruguen}:

\smallskip
{\em A quasi-split reductive linear algebraic group $G$
over $k$ is (1)-special if and only if there exists an exact sequence of the form
\[  \xymatrix{  1 \ar@{->}[r] & H \ar@{->}[r]  &  G  \ar@{->}[r] & T \ar@{->}[r] & 1,} \]
where $T$ is a (1)-special $k$-torus, $H = H_1 \times \ldots \times H_n$, and each $H_i$ is of the form $R_{l_i/k}(\SL_{m_i})$ or $R_{l_i/k}(\Sp_{2n_i})$, for some field extension $l_i/k$ of finite degree.}

\smallskip
By Lemma~\ref{lem.torus2}, $T$ is (3)-special. We claim that $H$ is also (3)-special. If we can prove this claim, then
applying Lemma~\ref{lem.torus1}(a) to the above sequence, we will be able to conclude that $G$ is (3)-special, as desired.

To prove the claim, recall that by Lemma~\ref{lem.examples}, $\SL_n$ and $\Sp_{2n}$ are (3)-special for every $n$. By Lemma~\ref{lem.torus1}(c),
each $H_i$ is (3)-special, and by Lemma~\ref{lem.torus1}(b), $H$ is (3)-special, as claimed.

\smallskip
{\bf Case 3:} $G$ is an arbitrary connected smooth linear algebraic group defined over a finite field $k$.
Let $U$ be the unipotent radical of $G$. Recall that $U$ is defined as the largest smooth connected normal unipotent subgroup of $G$.
Since $k$ is perfect, $U$ is defined over $k$ and is $k$-split; 
see~\cite[Corollary V.15.5(ii)]{borel}. The quotient $\overline{G} = G/U$ is a reductive group over $k$. By~\cite[Lemma 1.13]{sansuc}, 
the natural morphism $H^1(K, G) \to H^1(K, \overline{G})$ is a bijection for any field $K$ containing $k$. By our assumption $G$ is (1)-special,
so $\Tors(K,G)=H^1(K, G) = 1$. Hence, $H^1(K, \overline{G}) = 1$ as well. We conclude that $\overline{G}$ is (1)-special. 

Now by Case 2, $\overline{G}$ is (3)-special. By
Lemma~\ref{lem.examples}(d), $U$ is also (3)-special. Applying Lemma~\ref{lem.torus1}(a) to the exact 
sequence $1 \to U \to G \to \overline{G} \to 1$, we conclude that $G$ is (3)-special. This completes the proof of Theorem~\ref{thm.main1}.
\qed

\begin{rmk} \label{rem.grothendieck} Our argument in Case 2
uses~\cite[Proposition 15]{huruguen}, whose proof, in turn, relies on Grothendieck's 
classification of special groups over an algebraically closed field~\cite[Theorem 3]{grothendieck-special}.
In using Grothendieck's classification,~Huruguen implicitly assumed that every (1)-special group $G$ over an algebraically 
closed field is (2)-special. This does not cause a problem though, either for us or in~\cite{huruguen}, since we 
established the equivalence of (1) and (2) for groups over an infinite field in Case 1
by a self-contained argument.
Alternatively, the equivalence of (1) and (2) over an infinite perfect field
can be deduced from the variant of the Grothendieck-Serre conjecture proved by Colliot-Th\'el\`ene and 
Ojanguren in~\cite[Theorem 3.2]{cto}. \qed
\end{rmk} 

\section{Proof of Corollary~\ref{cor.main3}}
\label{sect.cor.main3}

Let $k$ be an infinite field, and $G$ be a linear algebraic group over $k$.
We may assume that $G$ is a closed subgroup of $\GL_n$. Let $X=G \! \setminus \! \GL_n$, $K = k(X)$ and $\pi_K$ be the restriction of $\pi$ to the generic point $\eta \colon \Spec(K) \to X$ of $X$. 
Then $\ed(\pi_K) = \ed(G)$; see~\cite[Proposition 3.11]{merkurjev-survey} or~\cite[Theorem 3.4]{reichstein}. This means that 
there exists an intermediate subfield $k \subset K_0 \subset K$ and a pull-back diagram
 \[ \xymatrix{ \GL_n \ar@{->}[d]_{\pi}  & \pi^{-1}(\Spec(K)) \ar@{->}[r] \ar@{_{(}->}[l] \ar@{->}[d]^{\pi_K} & T_0 \ar@{->}[d]   \\
 X  &  \Spec(K) \ar@{_{(}->}[l]^{\eta \quad } \ar@{->}[r]  &  \Spec(K_0) } \]
such that $T_0 \to \Spec(K_0)$ is a $G$-torsor and $\trdeg_k(L_0) = \ed(G)$. Since $K$ is a finitely generated field extension of $k$, so is $K_0$.
In other words, there exists a dense open subvariety $X_0 \subset X$ defined over $k$, such that over $X_0$, $\pi$ is the pull-back of $\tau$, via a Cartesian diagram 
\[ \xymatrix{  \GL_n \ar@{->}[d]_{\pi}  & U_0 \ar@{->}[d]_{\pi} \ar@{_{(}->}[l]_{\text{open}}  \ar@{->}[rr]^{\phi} &  & W \ar@{->}[d]_{\nu}  \\
                 X & \ar@{_{(}->}[l]_{\text{open}}  X_0 \ar@{->}[rr]^{\overline{\phi}} & & Y, }  
\]
where $U_0 = \pi^{-1}(X_0)$, $Y$ is a geometrically integral scheme of finite type over $k$ of dimension $\dim_k(Y) = \ed(\pi) = \ed(G)$, the function field of $Y$ is $K_0$, $\nu$ is a $G$-torsor and the map $\overline{\phi}$ is dominant. 
Note that since $k$-points are dense in $\GL_n$, they are also dense in $W$. Now using Lemma~\ref{lem.gln}(b), as we did in the proof Theorem~\ref{thm.main2},
we see that for every semi-local ring $S$ and every torsor $\tau \colon T \to \Spec(S)$, $\tau$ can be obtained by pull-back from $\nu$. 
\qed

\begin{rmk} The above argument shows that the $G$-torsor $\nu \colon W \to Y$ in the statement of Corollary~\ref{cor.main3} 
can be chosen to be versal and independent of the choice of $S$ or $\tau$.  Here ``versal" means ``(1)-versal", ``(2)-versal" or ``(3)-versal"; these notions are equivalent
by Theorem~\ref{thm.main2}.
\end{rmk}

\section{Some evidence for Conjecture~\ref{conj.main4}}
\label{sect.conj.main4}
 
Let $k$ be an algebraically closed field. Our main observation is the following.

\begin{rmk} Conjecture~\ref{conj.main4} 
holds if (a) $\ed(\tau_K) = 0$ or (b) $\ed(\tau_K) = \ed(G)$.

\smallskip
Indeed, in case (a) Conjecture~\ref{conj.main4} reduces to the variant of the Grothendieck-Serre Conjecture~\ref{conj.g-s} proved in~\cite[Theorem 3.2]{cto} and in case
(b) Conjecture~\ref{conj.main4} reduces to Corollary~\ref{cor.main3}. In particular, if 
\begin{equation} \label{e.ed}
\text{$\ed(\alpha) = 0$ or $\ed(G)$}
\end{equation} 
for every field $K$ containing $k$ and every $G$-torsor $\alpha \colon X \to \Spec(K)$, then Conjecture~\ref{conj.main4} is satisfied
for every local ring $R$ and every $G$-torsor $T \to \Spec(R)$.
\end{rmk}

Condition~\eqref{e.ed} is obviously satisfied if $\ed(G) = 0$ (i.e., $G$ is a special group; see Corollary~\ref{cor.special-ed}) or $\ed(G) = 1$. Other examples are given below.
For simplicity, we will assume that $\cha(k) \neq 2$ or $3$.

\begin{prop} Conjecture~\ref{conj.main4} holds if 

\smallskip
(a) $G$ is the projective linear group $\PGL_n$, for $n = 2$, $3$ or $6$, or 

\smallskip
(b) $G$ is the exceptional group $G_2$. 
\end{prop}

\begin{proof}
In view of~\eqref{e.ed}, it suffices to show that $H^1(K, G) = 1$ for every field extension $K/k$ of transcendence degree $< \ed(G)$.

(a) Here $\ed(G) = 2$; see~\cite[Lemma 9.4]{reichstein}. Moreover, $H^1(K, \PGL_n)$ is in a natural bijective correspondence with isomorphism classes of central simple algebras of degree $n$ over $K$. By Tsen's theorem, $H^1(K, \PGL_n) = 1$ for any $K/k$ of transcendence degree 
$\leqslant 1$.

(b) Recall that $H^1(K, G_2)$ is in a bijective correspondence with isomorphism classes of $3$-fold Pfister forms $\langle \langle a, \, b, \, c \rangle \rangle$ over $K$
and $\ed(G_2) = 3$. Let $K/k$ be a field extension of transcendence degree $\leq 2$. By the Tsen-Lang theorem every $3$-fold Pfister form 
is isotropic and hence, hyperbolic over $K$; cf.~\cite[Theorem 11.2]{reichstein} and its proof. Thus $H^1(K, G_2) = 1$, as desired. 
\end{proof}
\section*{Acknowledgements} The authors are grateful to Shamil Asgarli, Michel Brion and Baptiste Morin for stimulating discussions and to the referee for a careful reading of the paper and constructive remarks. 

\emergencystretch=2em

\emergencystretch=2em
\renewcommand{\markboth}[2]{}
\end{document}